\documentclass[11pt,reqno]{amsart}
\usepackage{ytableau, tikz,hyperref}
\usepackage[backend=bibtex, style=alphabetic]{biblatex}
\usepackage{palatino}
\usepackage[left=3cm,right=3cm,top=3cm,bottom=3cm]{geometry}
\usepackage[shortlabels,inline]{enumitem}
\newcommand\characx{\mathbf{x}}
\newcommand{\key}{\kappa}
\newtheorem{theorem}{Theorem}
\newtheorem{remark}{Remark}
\newcommand\bremark{\begin{remark}\begin{upshape}}
\newcommand\eremark{\end{upshape}\end{remark}}
\newtheorem{proposition}{Proposition}
\newtheorem{corollary}{Corollary}
\newtheorem{lemma}{Lemma}

\DeclareMathOperator{\Tab}{Tab}
\DeclareMathOperator{\Mat}{Mat}

\allowdisplaybreaks
\title[ Demazure crystal structure for Flagged reverse plane partitions]{Demazure crystal structure for\\ Flagged reverse plane partitions}

\author[]{Siddheswar Kundu}
\address{The Institute of Mathematical Sciences, A CI of Homi Bhabha National Institute, Chennai 600113, India}
\email{siddheswark@imsc.res.in}
\thanks{The author acknowledges partial funding from a DAE Apex Project grant to the Institute of Mathematical Sciences, Chennai.}
\begin{document}
\begin{abstract}
Given a skew shape $ \lambda / \mu $ and a flag $\Phi,$ we show that the set of all flagged reverse plane partitions of shape $\lambda / \mu$ and flag $\Phi$ is a disjoint union of Demazure crystals (up to isomorphism). As a result, the flagged dual stable Grothendieck polynomial $ g_{\lambda/\mu}(X_\Phi)$ is shown to be key positive.
\end{abstract}
\maketitle
\section{Introduction}
 The polynomial ring $\mathbb{Z}[x_1,x_2,\cdots]$ has a
 $\mathbb{Z}$-basis of key polynomials $\{\kappa_{\alpha}\},$ where $\alpha$ varies over all compositions, see \cite[Corollary 7]{RS}. For a permutation $\sigma$ in the symmetric group $\mathfrak{S}_n$ and a partition $\lambda $ with at most $n$ parts, the key polynomial $\kappa_{\sigma.\lambda}$ ($\sigma $ acts on $\lambda$ by usual left action of $\mathfrak{S}_n$ on $n$-tuples) is the character of the Demazure crystal $ \mathcal{B}_{\sigma} (\lambda)$. A polynomial $f \in \mathbb{Z}[x_1,x_2,\cdots] $ is called \emph{key positive} if it is a sum of key polynomials. For example, Schubert polynomial
is key positive. 

Given a skew shape $\lambda/ \mu $ and a flag $\Phi,$ Reiner-Shimozono \cite[Theorem 20]{RS} has given an expansion of a flagged skew Schur polynomial $s_{\lambda / \mu}(X_{\Phi})$ as a sum of key polynomials. A crystal theoretic statement of this result is namely $\Tab(\lambda/\mu, \Phi)$, i.e., the set of all flagged semi-standard tableaux of shape $\lambda /\mu$
and flag $\Phi$ is a disjoint union of Demazure crystals. In ~\cite[Theorem 4 \& Appendix]{krsv} this theorem is proved. In this paper, we extend this theorem as follows:
\begin{theorem}
  Let $ \lambda, \mu \in \mathcal{P}[n]$ such that $\mu \subset \lambda$ and $ \Phi $ be a flag in $\mathcal{F}[n]$. Then $\Re(\lambda/\mu,\Phi) $ is a disjoint union of Demazure crystals (up to isomorphism). More precisely,
  $$ \Re(\lambda/\mu,\Phi) \cong \displaystyle\bigsqcup_{Q} \mathcal{B}_{\tau} (\widehat{\beta(Q)}^{\dagger}),$$
  where $Q$ varies over all $(\lambda/\mu, \Phi)$-compatible tableaux and $\tau$ is any permutation such that $\tau. \widehat{\beta(Q)}^{\dagger}=\widehat{\beta(Q)}$. See Section $2$ and Section $4$ for the notations.
\label{th 1}
\end{theorem}
As an important corollary, it provides an expansion of the flagged dual stable Grothendieck polynomial $ g_{\lambda/\mu}(X_\Phi) $ as a sum of key polynomials as follows: 
\begin{corollary}
 The flagged dual stable Grothendieck polynomial $ g_{\lambda / \mu}(X_{\Phi})$ is key positive. Explicitly, 
 $$ g_{\lambda / \mu}(X_{\Phi}) = \sum_{Q} \key_{\widehat{\beta(Q)}} ,$$ where $Q$ varies over all $(\lambda/\mu, \Phi)$-compatible tableaux.
\end{corollary}
One of the beautiful ways to show that a polynomial $g$ in $\mathbb{Z}[x_1,x_2,\cdots]$ is key positive is to find a combinatorial model $\mathcal{G}$ whose character is $g$ and prove that $\mathcal{G}$ is a disjoint union of Demazure crystals. For example, Sami Assaf and Anne Schilling have proved the key positivity of the Schubert polynomial in this approach, see \cite{Schubert}. Also, the key positivity of the flagged skew Schur polynomial $s_{\lambda / \mu}(X_{\Phi})$ have been proved in a similar approach in \cite{RS}, \cite{krsv}. We also follow this approach in this paper. 

This paper is organized as follows. In Section $2,$ we define flagged reverse plane partition$, $ flagged dual stable Grothendieck polynomial $g_{\lambda/\mu}(X_\Phi),$ key polynomial and recall two important theorems of the Burge correspondence. Section $3$ contains the definitions of type $A_{n-1}$ crystal, crystal structure on $\Re ( \lambda / \mu, m), $ and Demazure crystal. We prove our main theorem, i.e., Theorem $1$ in Section $4$ and also give an alternative proof (without giving an explicit decomposition) in Section $5$. 

\section{Preliminaries}
A \emph{partition} $ \lambda=(\lambda_1,\lambda_2,\cdots)$ is defined as a weakly decreasing sequence of non-negative integers that has only finitely many non-zero parts. The \emph{weight} of $\lambda$ is the sum of its parts, which we denote by $|\lambda|$. Given a positive integer $n,$ $\mathcal{P}[n]$ denotes the set of all partitions with at most $n$ parts. The \emph{Young diagram} associated with the partition $\lambda$ is a collection of boxes that are top and left justified, where the $i^{th}$ row contains $\lambda_i$ boxes. Through the misuse of notation, we denote the Young diagram of $\lambda$ again by $\lambda$. For Young diagrams $\lambda$ and $\mu$, where $ \mu \subset \lambda $ (i.e., $\mu_i \leq \lambda_i \text{ }\forall i $), the skew diagram $\lambda / \mu$ is obtained by eliminating the boxes of $\mu $ from those of $\lambda$.
\subsection{ Reverse plane partition}
Consider $\lambda, \mu \in \mathcal{P}[n] $ such that $\mu \subset \lambda$. A \emph{reverse plane partition} of skew shape $\lambda/\mu$ is a filling of the skew diagram $\lambda/\mu$ with positive integers which is weakly increasing along both rows and columns. Obviously, a skew semi-standard tableau of shape $\lambda/\mu$ is a reverse plane partition of the same shape. We denote by $\Re(\lambda/\mu, m)$ the set of all reverse plane partitions of shape $\lambda/\mu$ with entries in $[m]=\{1,2,\cdots,m\}$.

For $R \in \Re(\lambda/\mu, m)$ we define the \emph{weight} of $R$ as $wt(R)=(r_1,r_2,\cdots,r_m),$ where $r_i$ is the number of columns of $R$ containing  $i$. Observe that this definition differs from the usual definition of weight of a semi-standard tableau $R$ where $r_i$ is the number of entries of $R$ equal to $i$. To each $T \in \Re(\lambda/\mu, m)$ we assign its row reading word $r_T$ as follows: omit all entries from $T$ which are equal to the entry immediately below it; then read $T$ from left to right and bottom to top. Also, the \emph{height} $h(T)$ of $T$ is defined as the sequence of positive integers whose $i^{th}$ part (from the left) is the row number of the $i^{th}$ letter (from the left) in $r_T$. Clearly, $h(T)$ is a weakly decreasing sequence. For example, see Figure \ref{Figure 1}. 
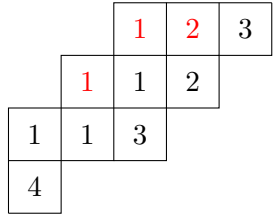
\begin{figure}
     \centering
\begin{tikzpicture}
    \draw (0,0) -- (0,1.4) -- (2.1, 1.4) -- (2.1, 2.8) -- (3.5,2.8);
    \draw (2.8,2.1)--(3.5,2.1)--(3.5,2.8);
    \draw (2.8,2.1)--(2.8,1.4)--(2.1,1.4);
    \draw (1.4,0.7)--(2.1,0.7)--(2.1,1.4);
    \draw (2.8,2.8) -- (2.8,2.1) -- (0.7, 2.1) -- (0.7, 0) -- (0,0);
    \draw (0, 0.7) -- (1.4, 0.7) -- (1.4, 2.8) -- (2.1, 2.8);
    \draw (0.35, 0.35) node {$4$};
    \draw (0.35, 1.05) node {$1$};
    \draw (1.05, 1.05) node {$1$};
    \draw (1.75, 1.05) node {$3$};
    \draw (1.05, 1.75) node {$\textcolor{red}{1}$};
    \draw (1.75, 1.75) node {$1$};
    \draw (2.45, 1.75) node {$2$};
    \draw (1.75, 2.45) node {$\textcolor{red}{1}$};
    \draw (2.45, 2.45) node {$\textcolor{red}{2}$};
    \draw (3.15, 2.45) node {$3$};
\end{tikzpicture}
     \caption{A reverse plane partition of shape $(5,4,3,1)/(2,1),$ weight $(3,1,2,1),$ row reading word $4113123 ,$ height $4333221$. The red entries do not contribute to the reading word.} 
     \label{Figure 1}
\end{figure}
\subsection{Flagged dual stable Grothendieck polynomial}
Lam and Pylyavskyy introduced dual stable Grothendieck polynomials in \cite{Lam}. The refined dual stable Grothendieck polynomials were introduced in \cite{Galashin_2016}. More generally, the row-flagged refined dual stable Grothendieck polynomials were studied in \cite{Kim}, see also \cite{hwang2021refined}. In this paper, we study a special case of the row-flagged refined dual stable Grothendieck polynomials, which we call the flagged dual stable Grothendieck polynomials.

A flag $\Phi$ = $(\Phi _1, \Phi _2, \cdots, \Phi_n)$ is defined as a weakly increasing sequence of positive integers, with the condition $\Phi_n = n$. $\mathcal{F}[n] $ denotes the set of all flags with $n$ parts. In the literature, it is not a requirement for a flag to possess the property of $\Phi_n = n$. However, for the purposes of our analysis, it is sufficient to consider flags that exhibit this characteristic.

Given $\lambda, \mu \in \mathcal{P}[n]$ and a flag $\Phi \in \mathcal{F}[n] ,$ $ \Re(\lambda/\mu,\Phi)$ denotes the set of all reverse plane partitions $R$ in $ \Re(\lambda/\mu, n)$ such that the entries in $i^{th}$ row of $R$ are $\leq \Phi_i $ for all $1 \leq i \leq n$. We define the \emph{flagged dual stable Grothendieck polynomial} $g_{\lambda/\mu}(X_\Phi)$ by 
$$ g_{\lambda/\mu}(X_\Phi)= \sum _{R} \characx^{wt(R)}, $$ where $R$ varies over $ \Re(\lambda/\mu,\Phi)$ and for $\alpha \in \mathbb{Z}^n _{+}\footnote{$ \mathbb{Z}^n _{+}$ is the set of all $n$-tuples of non-negative integers.},\text{ we let }  \characx^{\alpha}=x_1^{\alpha_1}x_2^{\alpha_2}\cdots x_n^{\alpha_n}$.
If $\Phi=(n,n,\cdots, n)$ then these polynomials become the dual stable Grothendieck polynomial $g_{\lambda/\mu}(x_1,x_2,\cdots,x_n)$ in \cite{Grothendieck}. Note that the flagged skew Schur polynomial $s_{\lambda/\mu}(X_\Phi)$ is the highest degree homogeneous component of $ g_{\lambda/\mu}(X_\Phi) $. 
\subsection{Key polynomial}Consider the polynomial ring $\mathbb{C}[\characx]$ in $n$ variables $\characx =(x_1,x_2,\cdots,x_n)$. Then for $i \in [n-1],$ the Demazure operators $T_i :\mathbb{C}[\characx] \rightarrow \mathbb{C}[\characx]$ 
are defined by: 
$$ (T_i f)(x_1, x_2, \cdots, x_n) = \frac{x_i\, f(x_1, x_2,  \cdots, x_n) - x_{i+1}\, f(x_1, \cdots, x_{i-1}, x_{i+1}, x_i, x_{i+2},  \cdots, x_n )}{x_i - x_{i+1}}$$

Given a permutation $w \text{ in the symmetric group }\mathfrak{S}_n$, we define $$T_w =  T_{i_1}T_{i_2} \cdots T_{i_k},$$ where $s_{i_1}s_{i_2} \cdots s_{i_k}$ is a reduced expression for $w$. Since $T_i$ satisfies the braid relations, $T_w$ is independent of the reduced expression.

For any $\alpha \in \mathbb{Z}^n _{+} ,$ the \emph{key polynomial} is defined by $\key_{\alpha}=T_w(\characx^{\alpha^{\dagger}}),$ where $\alpha^{\dagger}$ is the partition obtained by sorting the parts of $\alpha$ into decreasing order and $w$ is any permutation in $ \mathfrak{S}_n$ such that $w.\alpha^{\dagger}= \alpha$.
\subsection{The Burge correspondence}We recall the following two theorems related to the Burge correspondence \cite{fulton:yt} which we will use in Section 4. For positive integers $m,n,$ $\Mat_{m \times n}(\mathbb{Z}_{+})$ is the set of all $m \times n$ matrices with entries in $ \mathbb{Z}_{+}$. We assign a matrix $A=(a_{ij}) \in \Mat_{m \times n}(\mathbb{Z}_{+})$ to a two-rowed array $w_A$ as follows:
$$ w_{A}=\begin{bmatrix}         
i_t\hspace{0.2cm} \cdots \hspace{0.2cm}i_2 \hspace{0.2cm} i_1 \\
j_t \hspace{0.2cm} \cdots \hspace{0.2cm} j_2 \hspace{0.2cm}j_1 \\
           
\end{bmatrix}$$
so that for any pair $(i,j) \in [m] \times [n],$ there are $a_{ij}$ columns in $w_A$ equal to $
\begin{bmatrix}
   i \\
   j
\end{bmatrix}
$ and those are ordered in the following way.
\begin{itemize}
    \item $i_t \geq \cdots \geq i_2 \geq i_1 \geq 1 $.
    \item $ i_{k+1} > i_k$ whenever $ j_{k+1}>j_k$.
\end{itemize}
Further, given a flag $ \Phi $ in  $ \mathcal{F}[n],$ if $i_k \leq \Phi_{j_k} \forall k $ then $\textbf{i}$ is called $(\textbf{j},\Phi)$-compatible as in \cite{RS}, where $\textbf{i}=i_1i_2\cdots i_t$ and $\textbf{j}=j_1j_2\cdots j_t$.

\noindent \textbf{Examples:}
\begin{enumerate}
    \item $ \text{Take } A=
\begin{pmatrix}
1 & 2 & 1\\
0 & 1 & 3
\end{pmatrix}
,$ then $w_A =
\begin{bmatrix}
2 & 2 & 2 & 2 & 1 & 1 & 1 & 1\\
2 & 3 & 3 & 3 & 1 & 2 & 2 & 3
\end{bmatrix}.$

\item $ \text{If } B=
\begin{pmatrix}
1 & 3 & 1\\
0 & 1 & 1\\
0 & 0 & 2
\end{pmatrix}
,$ then $w_B =
\begin{bmatrix}
3 & 3 & 2 & 2 & 1 & 1 & 1 & 1 & 1\\
3 & 3 & 2 & 3 & 1 & 2 & 2 & 2 & 3
\end{bmatrix}.$
 In this case \textbf{i} is $ (\textbf{j}, (1,2,3))$-compatible, where $\textbf{i}=111112233$ and $\textbf{j}=322213233 $.

\end{enumerate}

\begin{theorem}
\begin{upshape}\cite[Appendix A, Proposition 2]{fulton:yt} \end{upshape}
    The Burge correspondence gives a one-to-one correspondence between $\Mat_{m \times n}(\mathbb{Z}_{+})$ and the set of pairs $(P, Q)$, where $P, Q$ are both semi-standard tableaux of same shape and entries of $P, Q$ are in $[n],[m]$ respectively.
\end{theorem}
If A corresponds to $(P,Q)$ then we write it by $(w_A \rightarrow \emptyset)=(P,Q)$.

\begin{theorem}
\begin{upshape}\cite[Appendix A, Symmetry Theorem $($b$)$]{fulton:yt} \end{upshape}
\label{burge}
    If $A$ corresponds to $(P,Q)$ then $A^t$ corresponds to $(Q,P)$.
\end{theorem}

\section{Crystal structure on $\Re(\lambda/\mu, m)$}
\subsection{Crystal of type $A_{n-1}$}A \emph{crystal} of type $A_{n-1}$ consists of an underlying non-empty finite set $\mathcal{B}$ along with the maps 
$$ e_i,f_i :\mathcal{B} \rightarrow \mathcal{B} \sqcup \{0\} \ \ \text{for $i \in [n-1],$} $$ 
$$ wt: \mathcal{B} \rightarrow \mathbb{Z}^{n},$$ where $0 \not \in \mathcal{B} \text{ is an auxiliary element},$ satisfying the following axioms:

\begin{enumerate}
    \item If $x,y \in \mathcal{B}$ then $e_i(x)=y \text{ if and only if } f_i(y)=x$. In this case, we further assume $wt(y)-wt(x)= \epsilon_i-\epsilon_{i+1},$ where $ \epsilon_i \in \mathbb{Z}^{n}$ whose $i^{th}$ entry is $1$ and others are $0$.
    \item $\phi_i(x) - \varepsilon _i (x) = wt(x) \cdot (\epsilon_i - \epsilon_{i+1})$ $\forall x \in \mathcal{B}$ and $i \in \{1, 2, \cdots , n-1\}$,  where $\varepsilon_i(x)$ (respectively $\phi_i (x)$) is the maximum number of times $e_i$ (respectively $f_i$) can be applied to $x$  so that the result is in $\mathcal{B}$. 
\end{enumerate}
The maps $e_i,f_i$ are called the raising and lowering operators respectively. By misuse of notation, a crystal is often denoted by its underlying set $\mathcal{B}$.

The \emph{character} of a crystal $\mathcal{B}$ is defined by $ch(\mathcal{B}) := \displaystyle \sum _{ u \in \mathcal{B}} {\characx}^{wt(u)}$.

If $\mathcal{B}$ is a crystal then we assign with $\mathcal{B}$ a directed graph which we call the \emph{crystal graph} of $ \mathcal{B}$. The vertex set of this graph is the underlying set $\mathcal{B}$ of the crystal and we draw an edge labelled by $i$ between two vertices $x,y$ starting at $x$ and ending at $y$ if $f_i(x)=y$.
We say $\mathcal{B}$ is \emph{connected} if the crystal graph (considered as an undirected graph) is connected. A subset $\mathcal{B}^{'}$ of a crystal $\mathcal{B},$ which is a union of connected components of $\mathcal{B},$ inherits a crystal structure from the crystal structure of $\mathcal{B}$. Here we say $ \mathcal{B'}$ a \emph{full subcrystal} of $\mathcal{B}$.

\noindent \textbf{Example:} The \emph{Standard} type $A_{n-1}$ crystal contains the underlying set $ \{1, 2, \cdots, n\}$ and the maps
\begin{equation*}
f_i(j) = \left\{
    \begin{array}{ll}
         i+1 & \quad \text{if} \quad j=i  \\
         0 & \quad \text{otherwise}
    \end{array}     \right. 
\text{ for $i \in [n-1]$} \quad  \text{ and } \quad wt(i) = \epsilon _i \quad \text{ for $i \in [n]$}.    
\end{equation*}
We will denote this crystal by $\mathbb{W}_n$.

\noindent \textbf{Tensor products of crystals:} If $\mathcal{A} \text{ and } \mathcal{B}$ are two type $A_{n-1}$ crystals then the tensor product crystal $\mathcal{A} \otimes \mathcal{B}$ is also defined. Its underlying set is $\{x \otimes y: x \in \mathcal{A}, y \in \mathcal{B}\}$. The raising and lowering operators are defined as follows:
\begin{equation*}\label{e action}
 e_i(x \otimes y) = \left\{ 
    \begin{array}{ll}
        e_i(x) \otimes y & \quad \text{if} \quad \varepsilon_i(x) > \phi_i(y)  \\
        x \otimes e_i(y) & \quad \text{if} \quad \varepsilon_i(x) \leq \phi_i(y) 
    \end{array}   \right. 
\end{equation*}
and 
\begin{equation*}\label{f action}
 f_i(x \otimes y) = \left\{ 
    \begin{array}{ll}
        f_i(x) \otimes y & \quad \text{if} \quad \varepsilon_i(x) \geq \phi_i(y)  \\
        x \otimes f_i(y) & \quad \text{if} \quad \varepsilon_i(x) < \phi_i(y) 
    \end{array}   \right. 
\end{equation*}
We define $wt(x \otimes y)=wt(x) + wt(y)$.

It is understood that $x \otimes 0 = 0 \otimes y = 0$. Tensor products are associative \cite[\S 2.3]{bump-sch}.
\noindent \subsection{Crystal structure on $ \Re(\lambda/\mu, m)$}Consider the $r$-fold tensor product $\mathbb{W}_m^{\otimes r}$ of the Standard type $ A_{m-1}$ crystal $\mathbb{W}_m$. Then the crystal structure of $\Re(\lambda/\mu, m)$ is given as follows:

Let $h$ be a finite sequence of positive integers. Let $ \Re(\lambda/\mu, m, h)$ denote the set of all reverse plane partitions in $ \Re(\lambda/\mu, m)$ with height $h$. Note that $ \Re(\lambda/\mu, m, h)$ could be empty. For example, $\Re(\lambda/\mu, m, h) $ is empty if we take $\lambda =(5,3), \mu=(2,1), h=221$ and $m$ is any positive integer. Thus we have the following:
$$ \Re(\lambda/\mu, m) = \displaystyle \bigsqcup_{h} \Re(\lambda/\mu, m, h), $$ where $h$ varies over all possible finite sequences. Then $ \Re(\lambda/\mu, m, h)$ is given a structure of type $A_{m-1}$ crystal via the following embedding into $\mathbb{W}_m^{\otimes k} $ $(k$ is the number of parts in $h)$:$$ T \mapsto r_T = w_1 w_2 \cdots w_k \mapsto w_1 \otimes w_2 \otimes \cdots \otimes w_k$$
The image under this embedding is a full subcrystal of $\mathbb{W}_m^{\otimes k}, $ see \cite{Grothendieck}.

\subsection{Demazure crystals}For $\sigma \in \mathfrak{S}_{n}, \lambda \in \mathcal{P}[n],$ the \emph{Demazure crystal} $\mathcal{B}_{\sigma} (\lambda)$ is defined as:
\begin{equation*} \label{eq:demcrys}
  \mathcal{B}_{\sigma} (\lambda) := \{f_{i_1} ^{k_1} f_{i_2} ^{k_2} \cdots f_{i_p} ^{k_p}  T_{\lambda} : k_j \geq0 \} \textcolor{black}{\setminus \{0\}}, 
\end{equation*}
where $s_{i_1} s_{i_2} \cdots s_{i_p}$ is any reduced expression of $\sigma$ and  $T_{\lambda}$ is the unique semi-standard tableau of shape and weight both equal to $\lambda$. 

It is a familiar fact that the Demazure crystal $\mathcal{B}_{\sigma} (\lambda)$ is independent of the choice of the reduced expression of $\sigma$.

\noindent \textbf{Example: } 
$\mathcal{B}_{s_{k-1} \cdots s_2s_1}((1))$ = $\mathcal{R}((1), k)$ for $ k \in [n]$ and this Demazure crystal is denoted by $\mathbb{B}_k$.
\section{proof of the main theorem}
In this section, we give an explicit decomposition of $\Re(\lambda/\mu,\Phi)$ into a disjoint union of Demazure crystals.

A word $y=y_1y_2 \cdots y_s$ is said to be \emph{Yamanouchi word} ~\cite[\S 5.2]{fulton:yt} if for all $t \geq 1,$  the number of $i'$s in $ y_t \cdots y_s$ is at least the number of $(i+1)'$s in it for all $i \geq 1$. For example, $321211$ is a Yamanouchi word but $ 1213$ is not. 

Let $\lambda, \mu \in \mathcal{P}[n]$ with $\mu \subset \lambda$ and $ \Phi \in \mathcal{F}[n]$. Then we call a semi-standard tableau $Q$ is $(\lambda/\mu , \Phi)$-compatible if  $\exists $ $T \in  \Re(\lambda/\mu,\Phi)$ such that $r_T$ is a Yamanouchi word together with $(
\begin{bmatrix}
           h(T) \\
           r_T \\
           
\end{bmatrix} \rightarrow \emptyset)
=(-, Q) 
$. In this case, $T$ is actually unique. Because if there exists another $T'$ then
$(
\begin{bmatrix}
           h(T)\\
           r_T \\
           
\end{bmatrix} \rightarrow \emptyset)
=(T_{\nu}, Q)= $ 
$(
\begin{bmatrix}
           h(T') \\
           r_{T'} \\
           
\end{bmatrix} \rightarrow \emptyset),$ where shape$(Q)=\nu$. Then by Lemma $6$ of \cite{Grothendieck} $T=T'$.

\noindent \textbf{Example:} Consider $\lambda=(4,4,3,2), \mu=(2,1), \Phi=(1,2,3,4)$. Then the tableau on the left-hand side in Figure \ref{fig:my_label} is $ (\lambda/\mu,\Phi) $-compatible and it corresponds to the reverse plane partition on the right-hand side in Figure \ref{fig:my_label}.
\begin{figure}
     \centering
    \begin{tikzpicture}
    \draw (0,0)--(0,2.8)--(2.1,2.8);
    \draw (0.7,0)--(0.7,2.8);
    \draw (1.4,1.4)--(1.4,2.8);
    \draw (0,1.4)--(2.1,1.4)--(2.1,2.8);
    \draw (0,2.1)--(2.1,2.1);
    \draw (0,0)--(0.7,0);
    \draw (0,0.7)--(0.7,0.7);
    \draw (0.35, 0.35) node {$4$};
    \draw (0.35, 1.05) node {$3$};
    \draw (0.35, 1.75) node {$2$};
    \draw (1.05, 1.75) node {$3$};
    \draw (1.75, 1.75) node {$4$};
    \draw (0.35, 2.45) node {$1$};
    \draw (1.05, 2.45) node {$2$};
    \draw (1.75, 2.45) node {$2$};
    \end{tikzpicture}
  \begin{tikzpicture}
  \draw (0,0) node {$ $};
  \draw (1.4,0) node {$ $};
  \end{tikzpicture}
  \begin{tikzpicture}
    \draw (0,0) -- (0,1.4) -- (2.1, 1.4) -- (2.1, 2.8) -- (2.8,2.8);
    \draw (2.8,2.1)--(2.8,1.4)--(2.1,1.4);
    \draw (1.4,0.7)--(2.1,0.7)--(2.1,1.4);
    \draw (2.8,2.8) -- (2.8,2.1) -- (0.7, 2.1) -- (0.7, 0) -- (0,0);
    \draw (0.7,0)--(1.4,0)--(1.4,0.7);
    \draw (0, 0.7) -- (1.4, 0.7) -- (1.4, 2.8) -- (2.1, 2.8);
    \draw (0.35, 0.35) node {$2$};
    \draw (1.05, 0.35) node {$4$};
    \draw (0.35, 1.05) node {$\textcolor{red}{2}$};
    \draw (1.05, 1.05) node {$2$};
    \draw (1.75, 1.05) node {$3$};
    \draw (1.05, 1.75) node {$1$};
    \draw (1.75, 1.75) node {$1$};
    \draw (2.45, 1.75) node {$2$};
    \draw (1.75, 2.45) node {$\textcolor{red}{1}$};
    \draw (2.45, 2.45) node {$1$};
\end{tikzpicture}
\caption{An example of $(\lambda/\mu,\Phi)$-compatible tableau. }
     \label{fig:my_label}
\end{figure}

Choose a $(\lambda/\mu , \Phi)$-compatible tableau $Q$ and let shape$(Q)=\nu$. Then we define the following:
$$\mathcal{C}(\lambda/\mu,Q,\Phi):= \{ T \in \Re(\lambda/\mu,\Phi) : (
\begin{bmatrix}
           h(T)\\
           r_T \\
           
\end{bmatrix} \rightarrow \emptyset)
=(-, Q) 
  \}.$$
  Thus $ \Re(\lambda/\mu,\Phi) = \displaystyle\bigsqcup_{Q} \mathcal{C}(\lambda/\mu,Q,\Phi) ,$ where $Q$ runs over all $(\lambda/\mu, \Phi)$-compatible tableaux. We will show that $\mathcal{C}(\lambda/\mu, Q,\Phi)$ is crystal isomorphic to $ \mathcal{B}_\sigma {(\nu)}$ for some permutation $\sigma,$ which means there is a weight-preserving bijection $\Psi: \mathcal{C}(\lambda/\mu, Q,\Phi) \rightarrow \mathcal{B}_\sigma {(\nu)}$ which intertwines with the raising and lowering operators.  
  \bremark
  For every $T \in \mathcal{C}(\lambda/\mu,Q,\Phi),$ $ h(T)$ are all same, say $h_0$ (see ~\cite[\S 2]{Grothendieck}).
  \eremark
           
For $\alpha \in \mathbb{Z}^n_{+} ,$ we denote by \textbf{b}($\alpha$) the word $ b^{(n)} \cdots b^{(2)}b^{(1)}$ in which $b^{(j)}$ consists of a string of $\alpha_j$ copies of $j$. For example, $\textbf{b}(3,1,4,0,2) =5533332111$.

Given $\alpha \in \mathbb{Z}^n_{+} $ and a flag $\Phi$ in $\mathcal{F}[n],$ we define $\mathcal{W}(\alpha, \Phi)$ as the set of all words $\textbf{v}= v_n \cdots v_2 v_1$ such that each $v_i$ is a maximal row word (i.e., the last letter of $v_i$ is greater than the first letter of $v_{i-1}$) of length $\alpha_i$ together with the following properties:
\begin{itemize}
    \item each letter in $v_i$ can be at most $\Phi_i$.
    \item $(
    \begin{bmatrix}
    \textbf{b}(\alpha)\\
    \textbf{v}
    \end{bmatrix} \rightarrow \emptyset) =(-$, $\text{key}(\alpha)) ,$ where $\text{key}(\alpha)$ is the unique semi-standard tableau of shape $\alpha^{\dagger}$ and weight $\alpha$.
\end{itemize}

Combining Theorem 12 and Theorem 13 of \cite{krsv} we have the following theorem.
\begin{theorem}
    For any $\beta \in \mathbb{Z}^n_{+} $ and a flag $\Phi \in \mathcal{F}[n],$ either $\mathcal{W}(\beta, \Phi)$ is empty or there is a one-to-one correspondence between $\mathcal{W}(\beta, \Phi)$ and $\mathcal{B}_{\sigma}(\hat{\beta}^{\dagger})$ via $ \textbf{v} \mapsto P(\textbf{v}) ,$ for some $\hat{\beta} \in \mathbb{Z}^n_{+} $ with $\beta^{\dagger} = \hat{\beta}^{\dagger}$. Here $\sigma$ is any permutation such that $\sigma. \hat{\beta}^{\dagger}=\hat{\beta}$ and $P(\textbf{w})$ denotes the unique semi-standard tableau Knuth equivalent to a given word \textbf{w}.
    \label{ABC}
\end{theorem}

\begin{proposition}
  There is a bijection $\Omega$ between $\mathcal{C}(\lambda/\mu,Q,\Phi)$ and $\mathcal{W}(\beta(Q), \Phi)$ such that if $T \mapsto \Omega(T)$ then $r_T$ and $\Omega(T)$ are Knuth equivalent, where $\beta(Q)$ is the weight of the left key tableau $K_{\_}(Q)$
\footnote{ For definition and computation of $ K_{\_}(Q)$ see \cite{RS}, \cite{plactic}, \cite{willis}.} of $Q$.
\label{Pro}
\end{proposition}
\begin{proof}

      This proposition is closely related to Proposition 6 in \cite{krsv}. Let $T \in \mathcal{C}(\lambda/\mu, Q,\Phi)$. Assume that $M(T)$ is the matrix that corresponds to the two-rowed array
    $
\begin{bmatrix}
            h_0 \\
            r_T \\
           
\end{bmatrix}$. Suppose that $
\begin{bmatrix}
           \text{rev}(\textbf{i}) \\
           \text{rev}(\textbf{a}) \\
           
\end{bmatrix}$ is the two-rowed array associated to $M(T)^t ,$ where rev($\textbf{i}$) is the word $\cdots i_2i_1$ if $\textbf{i}=$$i_1i_2 \cdots $. Thus \textbf{i} is $(\textbf{a}, \Phi)$-compatible because $T$ $\in \Re(\lambda/\mu, \Phi)$ and by definition of $h(T)$. Then $(
\begin{bmatrix}
           h_0 \\
           r_T \\
           
\end{bmatrix} \rightarrow \emptyset)
=(P(r_T), Q) \implies 
$
$(
\begin{bmatrix}
           \text{rev}(\textbf{i}) \\
           \text{rev}(\textbf{a}) \\
           
\end{bmatrix} \rightarrow \emptyset)=(Q,P(r_T))
$ (by Theorem \ref{burge}).

Let $\textbf{a}'$ be the unique word such that 
$(
\begin{bmatrix}
           \text{rev}(\textbf{i}) \\
           \text{rev}(\textbf{a}') \\
           
\end{bmatrix} \rightarrow \emptyset)=(K_{\_}(Q),P(r_T))$. 
Then by Lemma 5 and Lemma 6 of \cite{krsv}, \textbf{i} is $(\textbf{a}', \Phi)$-compatible. Let $ \begin{bmatrix}
           \text{rev}(\textbf{j}) \\
           \text{rev}(\textbf{u}) \\
           
\end{bmatrix} $
be the two-rowed array corresponding to the matrix $B$ so that $ B^t$ corresponds to
$
\begin{bmatrix}
           \text{rev}(\textbf{i}) \\
           \text{rev}(\textbf{a}') \\
           
\end{bmatrix}
$. Thus $(
\begin{bmatrix}
           \text{rev}(\textbf{j}) \\
           \text{rev}(\textbf{u}) \\
           
\end{bmatrix} \rightarrow \emptyset)=(P(r_T),$ $ K_{\_}(Q))
$ (by Theorem \ref{burge}).
Therefore by Corollary 12 of \cite{RS}, we have rev(\textbf{u}) $ \in  \mathcal{W}(\beta(Q), \Phi)$. We define $\Omega(T) = $ rev$(\textbf{u})$.

Let $T,T' \in \mathcal{C}(Q,\lambda/\mu,\Phi)$ such that $\Omega(T)= \Omega(T') \implies \text{rev}(\textbf{u})= \text{rev}(\textbf{u}')  $. This implies $P(r_{T})=P(r_{T'})$. So $ \begin{bmatrix}
           h_0 \\
           r_{T} \\
           
\end{bmatrix} =
\begin{bmatrix}
           h_0 \\
           r_{T'}\\
           
\end{bmatrix}$.
Then Lemma $6$ in \cite{Grothendieck} gives $T=T'$. Thus $\Omega$ is injective. Let $\textbf{s} \in \mathcal{W}(\beta(Q), \Phi) $. Then by Theorem \ref{ABC}, $P(\textbf{s}) \in \mathcal{B}_{\tau}(\widehat{\beta(Q)}^{\dagger}),$ where $\tau$ is any permutation such that $\tau.
\widehat{\beta(Q)}^{\dagger}=\widehat{\beta(Q)}$. Note that $ \widehat{\beta(Q)}^{\dagger} =$ shape$(Q)=\nu$. Let $s_{i_1} s_{i_2} \cdots s_{i_q}$ be any reduced expression of $\tau$. Then $P(\textbf{s})= f_{i_1} ^{k_1} f_{i_2} ^{k_2} \cdots f_{i_q} ^{k_q} T_{\nu}$ for some $ (k_1,k_2,\cdots,k_q) \in \mathbb{Z}^q _{+}$. Take $f=f_{i_1} ^{k_1} f_{i_2} ^{k_2} \cdots f_{i_q} ^{k_q}$. By definition of $Q,$ there exists a unique $T_0 \in \Re (\lambda/\mu,\Phi)$ such that $r_{T_0}$ is a Yamanouchi word of weight $\nu$ and 
$(\begin{bmatrix}
           h_0 \\
          r_{T_0} \\
\end{bmatrix} \rightarrow \emptyset)= (T_{\nu}, Q)$. Then using Lemma $7$ in \cite{Grothendieck} and Proposition 29 in \cite{RS} we can say $f(T_0) \in \mathcal{C}(Q,\lambda/\mu,\Phi)$ such that $\Omega(f(T_0))= \textbf{s}$. Therefore $\Omega$ is surjective. Clearly, $r_T$ and $\Omega(T) $ are Knuth equivalent.
\end{proof}

The following proposition asserts our main theorem. 
\begin{proposition}
 $ \Psi: \mathcal{C}(Q,\lambda/\mu,\Phi) \rightarrow  \mathcal{B}_{\tau}(\widehat{\beta(Q)}^{\dagger})$ via $T \mapsto P(\Omega(T))$ is a weight-preserving bijection which intertwines the crystal raising and lowering operators. Here $\tau$ is any permutation such that $\tau.
 \widehat{\beta(Q)}^{\dagger}=\widehat{\beta(Q)}$.
\end{proposition} 
\begin{proof}
           
           
By Theorem \ref{ABC} and Proposition \ref{Pro}, $ \Psi$ is a bijection. Evidently, $\Psi$ is a weight-preserving. Also, it follows from Lemma 7, Corollary 8 of \cite{Grothendieck} and characteristics of Knuth equivalence that $\Psi$ commutes with the raising and lowering operators.
\end{proof}
\section{Alternative proof }
In this section, we give an alternative proof of Theorem \ref{th 1} without giving an explicit decomposition. For this proof we use Lemma $1$ in \cite{krsv} which says the following.
\begin{lemma}
\label{lemma}
    For a flag $\Phi \in \mathcal{F}[n]$ and $ \rho \in \mathbb{Z}_+ ^n,$ $\mathbb{B}_{\Phi} ^{\rho}$ is a disjoint union of Demazure crystals, where
    $$\mathbb{B}_{\Phi}  ^{\rho}:= \mathbb{B}_{\Phi_n} ^{\otimes \rho_n} \otimes \cdots \otimes  \mathbb{B}_{\Phi_2} ^{\otimes \rho_2} \otimes \mathbb{B}_{\Phi_1}^{\otimes \rho_1}.$$
\end{lemma}
This lemma is an application of \cite[Theorem 1.2]{sami}.
\bremark
Here we take the definition of $\mathbb{B}_{\Phi} ^{\rho}$ to be slightly different from the definition given in \cite{krsv} because we are using the tensor product rule from ~\cite[\S 2.3]{bump-sch} instead of the rule given in  ~\cite[\S 3.1]{krsv}. It is done because we are using the definition of reading word from \cite{Grothendieck}. 
\eremark

\noindent \textbf{Alternative proof:}
We know that the height $h$ of an element in $ \Re(\lambda/\mu, m)$ is a weakly decreasing sequence of positive integers. Thus $h \text{ can be written as } h=h_1^{t_1}h_2^{t_2} \cdots h_s^{t_s}, $ where $n \geq h_1 > h_2 > \cdots > h_s \geq 1$ and each $h_i$ occurs $t_i$ times in $h$. Then 
$ \mathbb{B}_{\Phi_{h_1}} ^{\otimes t_1} \otimes \mathbb{B}_{\Phi_{h_2}} ^{\otimes t_2} 
 \otimes \cdots \otimes \mathbb{B}_{\Phi_{h_s}} ^{\otimes t_s}$ is a disjoint union of Demazure crystals by the above lemma.

Let $ \Re(\lambda/\mu, \Phi , h)$ denote the set of all elements in $ \Re(\lambda/\mu, \Phi)$ that possess height $h$. Then $ \Re(\lambda/\mu, \Phi, h)$ is a disjoint union of Demazure crystals because
of the following two facts:
\begin{itemize}
    \item $ \Re(\lambda/\mu, \Phi, h)=( \mathbb{B}_{\Phi_{h_1}} ^{\otimes t_1} \otimes \mathbb{B}_{\Phi_{h_2}} ^{\otimes t_2} 
 \otimes \cdots \otimes \mathbb{B}_{\Phi_{h_s}} ^{\otimes t_s})$ $ \cap$ $ \Re(\lambda/\mu, n, h)$.

 \item  $ \Re(\lambda/\mu, n, h) $ is a full subcrystal of  $\mathbb{W}_n^{\otimes k},$ where $ k=t_1 + t_2 +\cdots + t_s $.
\end{itemize}
  Thus $\Re(\lambda/\mu, \Phi) $ is a disjoint union of Demazure crystals because
   $$ \Re(\lambda/\mu, \Phi) = \displaystyle \bigsqcup_{h} \Re(\lambda/\mu, \Phi, h), $$ where $h$ varies over all possible finite sequences.
\bremark
It is easy to see that $ \Tab(\lambda/\mu, \Phi)= \Re(\lambda/\mu, \Phi, h), \text{ where } h=\textbf{b}(\delta), \delta =\lambda-\mu$. Therefore we can deduce Theorem 4 in \cite{krsv} from the above discussion.
\eremark
\section*{Acknowledgements} The author would like to express his gratitude to Sankaran Viswanath for providing insightful comments and suggestions. The author is also grateful to Travis Scrimshaw for providing useful comments on the earlier preprint of this paper. 
\printbibliography
\end{document}